\documentclass[11pt]{article}
\usepackage{latexsym,amssymb,amsmath,amsfonts,amsthm}
\usepackage{graphics,graphicx,mathrsfs,subfigure}
\usepackage{color,xspace}
\newcommand{\R}{{\mat R}}

\newcommand{\be}{\begin{eqnarray}}
\newcommand{\ben}{\begin{eqnarray*}}
\newcommand{\en}{\end{eqnarray}}
\newcommand{\enn}{\end{eqnarray*}}

\newcommand{\mat}{\mathbb}

\newtheorem{theorem}{Theorem}[section]
\newtheorem{lemma}[theorem]{Lemma}

\newtheorem{remark}[theorem]{Remark}

\definecolor{rot}{rgb}{1,0,0}
\definecolor{hw}{rgb}{0,0,1}

\begin{document}
\renewcommand{\theequation}{\arabic{section}.\arabic{equation}}
\title{\bf

Observability inequality for the wave equation
 
}
\author{ Suliang Si\thanks{School of Mathematics and Statistics, Shandong University of Technology,
Zibo, 255000, China ({\tt sisuliang@amss.ac.cn})}}
\date{}

 



\maketitle

\begin{abstract}
 In this paper,  only Carleman estimates are used, without energy estimates, we derive observability inequality. The main tool consists in the use of a new Carleman estimate.
 
\end{abstract}

%


\section{Introduction}

Let \(\Omega\subset\R^n\) be a bounded domain with \(\mathcal{C}^2\)-boundary \(\partial\Omega\). We investigate the following hyperbolic initial boundary value problem (IBVP)
\begin{equation}\label{u}
\begin{cases}
\partial_t^2 u(x,t)-\Delta u(x,t)+q(x)u(x,t) =F(x,t),  & (x,t) \in Q:=\Omega\times(0,T), \\ 
u(x,0)=u_0, \quad  \partial_tu(x,0)=u_1,  & x \in \Omega,\\
u(x,t)=0, & (x,t) \in \partial\Omega\times(0, T). \\ 
\end{cases}
\end{equation}
First of all, suppose that \( u_0 \in H^1(\Omega), \, u_1 \in L^2(\Omega), \, q \in L^\infty(\Omega)\) and \( F \in L^1(0, T; L^2(\Omega)) \) satisfying the compatibility condition \( u_0(x)=0 \) for all \( x \in \partial\Omega \). Then the IBVP (\ref{u}) is well-posed and one can also prove  that
\[
u \in C([0, T]; H^1_0(\Omega)) \cap C^1([0, T]; L^2(\Omega)).
\]
Assume  there exists \( x_0\notin \Omega\) such that 
\[\Gamma_0\supset\{x\in\partial\Omega; \ (x-x_0)\cdot\nu\geq 0\}.\]
Define \[\Sigma_0=\Gamma_0\times(0,T).\]
We state the following classical theorem.
\begin{theorem}

Let \(T>3\sup\limits_{x\in\Omega} |x - x_0|\). Then for any weak solution
\[u \in C([0, T]; H^1_0(\Omega)) \cap C^1([0, T]; L^2(\Omega))\]
of (\ref{u}), we have
\begin{equation}\label{Op}
\|u_0\|_{H_0^1(\Omega)}^2 + \|u_1\|_{L^2(\Omega)}^2 
\le M\bigl(\|\partial_\nu u\|_{L^2(\Sigma_0)}^2 + \|F\|_{L^2(Q)}^2\bigr)
\end{equation}
for any \(u_0\in H_0^1(\Omega)\) and \(u_1\in L^2(\Omega)\).
Here the constant \(M>0\) is independent of \(u_0, u_1\).

\end{theorem}
Inequality (\ref{Op}) is called an observability inequality.
Observability inequalities stand as the fundamental quantitative tool in the control theory of infinite-dimensional distributed parameter systems governed by partial differential equations (PDEs). Formally, an observability inequality establishes a uniform reverse norm bound: the full state energy norm over the entire spatial domain $\Omega$ can be dominated by the norm of measurements collected only on a partial boundary segment $\Gamma_0\subset\partial\Omega$ or interior subdomain $\omega\subset\Omega$ within a finite time interval $(0,T)$. Beyond control design, these estimates encode the quantitative unique continuation property ($\mathcal{UCP}$), serve core analytical machinery for inverse problems, boundary stabilization, and semilinear PDE controllability.

The  observability inequalities originates from finite-dimensional linear state-space control theory proposed by Kalman \cite{K60}.  In the early 1970s, Russell \cite{Russell78} extended the duality principle to hyperbolic wave equations, conjecturing that boundary exact controllability relies on an analogous trace inverse estimate for adjoint wave solutions, though rigorous functional formulations were absent at that stage.  Ho presented the first complete functional setup for boundary wave observation and unified notational conventions for subsequent PDE control research.

The foundational qualitative theory matured between 1986 and 1995, spearheaded by Lions’ Hilbert Uniqueness Method (HUM) \cite{Lions88}. Lions formalized the defining duality equivalence that remains central to PDE control: a linear PDE system is exactly boundary controllable over time $T$ if and only if its dual adjoint system satisfies an observability inequality. Lions and Komornik \cite{Komornik94} developed the multiplier method—the first systematic analytical tool to validate observability bounds for constant-coefficient wave equations on smooth $\mathcal{C}^2$ domains. The standard multiplier proof pipeline integrates energy identities, compact Sobolev embeddings, and contradiction arguments premised on exogenous global $\mathcal{UCP}$. A critical inherent limitation of this qualitative framework is that it only guarantees existence of some positive observability constant $C>0$, without any explicit dependence on potential, damping, or drift coefficients embedded in the PDE.

A pivotal geometric complement emerged via Bardos, Lebeau and Rauch \cite{BLR92}, who derived the celebrated Geometric Control Condition (GCC) using microlocal wavefront propagation analysis. GCC states that observability holds if every geometric optical ray inside $\Omega$ intersects the observation boundary $\Gamma_0$ within time $T$, providing nearly necessary and sufficient geometric constraints for smooth constant-coefficient wave systems. Nevertheless, GCC still depends on pre-assumed $\mathcal{UCP}$ and fails to deliver quantitative growth laws for $C$. Yao \cite{Yao99} later generalized multiplier constructions to Riemannian manifolds and variable principal coefficients, yet retained non-constructive contradiction logic without explicit tracking of the observability constant.

Parallel to the multiplier-microlocal school, researchers developed an independent analytical route built upon Carleman weighted pointwise estimates, originally invented by Carleman \cite{Carleman39} in 1939 to prove elliptic unique continuation. Bukgeim and Klibanov transplanted Carleman’s exponential weight technique to hyperbolic ill-posed inverse problems in the 1980s. Fursikov and Imanuvilov \cite{FI94} pioneered Carleman-based observability bounds for parabolic heat equations, circumventing strict geometric convexity constraints required by multiplier arguments. Tataru \cite{Tataru95} further advanced global spacetime pseudoconvex weights for wave equations between 1995 and 1996, accommodating time-dependent, low-regularity lower-order terms. Despite these advances,  Carleman derivations shared two unresolved bottlenecks: global $\mathcal{UCP}$ was treated as an independent prerequisite rather than a corollary of the inequality, and lower-order perturbation terms were absorbed via non-quantitative contradiction steps that obscured closed-form expressions for $C$. As emphasized by Zuazua \cite{Zuazua93}, quantitative tracking of observability constants with respect to linearized potentials is mandatory for fixed-point arguments on semilinear PDEs, forming a prominent open problem throughout the 1990s.

A paradigm-shifting quantitative breakthrough arrived with Zhang \cite{Zhang00a}. The paper unified pointwise Carleman identities and global energy estimates into a direct integration scheme that discarded compactness-based contradiction entirely, eliminating the requirement of a priori $\mathcal{UCP}$—unique continuation emerges as a natural byproduct of the derived inequalities. For wave equations with general time-variant, non-smooth lower-order terms
\[
w_{tt}-\Delta w = q_1 w + q_2 w_t + \langle q_3,\nabla w\rangle,\quad w|_{\Sigma}=0,
\]
two canonical boundary observability inequalities were established with fully explicit exponential constant growth:
\[
\|w_0\|_{H_0^1(\Omega)}^2+\|w_1\|_{L^2(\Omega)}^2 \le C(r)\left\|\partial_\nu w\right\|_{L^2(\Sigma_0)}^2,\quad C(r)=C\exp\big(Cr^2\big),
\]
where $r=\|q_1\|_{L^{n+1}(Q)}+\|q_2\|_{L^\infty(Q)}+\|q_3\|_{L^\infty(Q)}$. A weaker dual version for $L^2\times H^{-1}$ initial data was also constructed for pure potential cases. This explicit quantitative framework immediately enabled rigorous high-dimensional semilinear exact controllability, filling a longstanding gap left by qualitative multiplier and early Carleman theories, and inaugurated the subfield of quantitative observability analysis \cite{Zhang00b}.

After 2000, observability inequality theory diversified rapidly across PDE families, geometric settings, and stochastic systems. The explicit Carleman machinery was extended to parabolic heat equations, fourth-order plate systems, elastic waves, Maxwell equations, and stochastic hyperbolic PDEs \cite{IY01,Zhang08}. Comparative work by Duyckaerts, Zhang, and Zuazua \cite{DZZ08} identified fundamental dimensional discrepancies: one-dimensional wave systems admit subexponential $\exp\big(C\sqrt{r}\big)$ observability constants, while dimensions $n\ge2$ yield quadratic exponential scaling, opening research on sharp optimal constant orders. Meanwhile, observability inequalities were recognized as equivalent to global stability estimates for inverse problems, supplying computable error bounds for source recovery and coefficient inversion \cite{Isakov06}. Modern extensions further cover discrete lattice PDEs, compact Riemannian manifolds, fractional evolution equations, and logarithmic weak observability for insufficient observation time or measure-zero sensor sets \cite{AEWZ14}.

Despite substantial progress, multiple core open frontiers persist: derivation of sharp, dimensionally optimal observability constants, extension of fully explicit Carleman frameworks to degenerate/singular PDE operators, unified quantitative analysis for coupled multi-physics systems, and numerically tractable constant computation for engineering PDE models. 

Compared with the available references \cite{Zhang00b,AEWZ14}, existing studies use Carleman estimates together with energy estimates at \(t=\frac{T}{2}\). Our novelty lies in the capability to deal with both estimates concurrently at \(t=0\).

\section{Carleman estimates}
Let \[\psi(x,t) = |x - x_0|^2 - \beta (t-t_0)^2 + \beta_0, \quad \beta\in(0,1)\]
where $\beta_0 > 0$ is chosen such that $\psi \geq 1$ in $\Omega \times (0, T)$.
We define the weight function
\[\varphi(x,t)=e^{\lambda \psi(x,t)},\]
where \(\lambda>0\) is a second large parameter. 
Throughout this article, \(M>0\) denote generic constants which are independent of
parameter \(s\).

A Carleman estimate is an $L^2$-weighted estimate for the wave operator, and is stated
as follows.
\begin{lemma}\label{Car}
If \(t_0>2\sup\limits_{x\in\Omega} |x - x_0|\),
then there exist $s_0 > 0$, $\lambda > 0$ and a positive constant $M$ such that for all $s \geq s_0$:
\begin{equation}\label{ine}
\begin{aligned}
&s^{1/2} \int_\Omega e^{2s\varphi(0)} |\partial_t v(0)|^2 dx+s^{1/2} \int_\Omega e^{2s\varphi(0)} |\nabla v(0)|^2 dx+s^{5/2} \int_\Omega e^{2s\varphi(0)} | v(0)|^2 dx \\
&+ s \int_{0}^T \int_\Omega e^{2s\varphi} \left( |\partial_t v|^2 + |\nabla v|^2 \right) dxdt + s^3 \int_{0}^T \int_\Omega e^{2s\varphi} |v|^2 dxdt \\
&\leq M \int_{0}^T \int_\Omega e^{2s\varphi} |\partial_t^2v-\Delta v+qv|^2 dxdt + M s \int_{0}^T \int_{\Gamma_0} e^{2s\varphi} |\partial_\nu v|^2 d\sigma dt\\
&+Ms  \int_\Omega e^{2s\varphi(T)} \left( |\partial_t v(T)|^2 + |\nabla v(T)|^2 \right) dx + s^3 \int_\Omega e^{2s\varphi(T)} |v(T)|^2 dx,
\end{aligned}
\end{equation}
for all $v \in C^1((0,T); H^1(\Omega))$ satisfying $(\partial_t^2-\Delta+q) v \in L^2(\Omega \times (0,T))$, $\partial_\nu v \in L^2(\partial\Omega \times (0,T))$.
\end{lemma}
Lemma \ref{Car} is a novel Carleman estimate.  Compared with the classical existing Carleman estimates,  we do not need to impose the assumptions \(v(x,0)=0\) for \(x\in\Omega\) or \(\partial_tv(x,0)=0\), \(x\in\Omega\).
If \(T\) is large enough, the terms of (\ref{ine}) at times \(t=T\)
 can be removed, which can be found in \cite{BBE13}, we omit it.
\begin{remark}
If \(T > 3\sup\limits_{x\in\Omega} |x - x_0|\) and \(t_0>2\sup\limits_{x\in\Omega} |x - x_0|\), 
then there exist $s_0 > 0$, $\lambda > 0$ and a positive constant $M$ such that for all $s \geq s_0$:
\begin{equation}
\begin{aligned}
&s^{1/2} \int_\Omega e^{2s\varphi(0)} |\partial_t v(0)|^2 dx+s^{1/2} \int_\Omega e^{2s\varphi(0)} |\nabla v(0)|^2 dx+s^{5/2} \int_\Omega e^{2s\varphi(0)} | v(0)|^2 dx\\
& + s \int_{0}^T \int_\Omega e^{2s\varphi} \left( |\partial_t v|^2 + |\nabla v|^2 \right) dxdt + s^3 \int_{0}^T \int_\Omega e^{2s\varphi} |v|^2 dxdt \\
&\leq M \int_{0}^T \int_\Omega e^{2s\varphi} |\partial_t^2v-\Delta v+qv|^2 dxdt + M s \int_{0}^T \int_{\Gamma_0} e^{2s\varphi} |\partial_\nu v|^2 d\sigma dt
\end{aligned}
\end{equation}
for all $v \in C^1((0,T); H^1(\Omega))$ satisfying $(\partial_t^2-\Delta+q) v \in L^2(\Omega \times (0,T))$, $\partial_\nu v \in L^2(\partial\Omega \times (0,T))$.
More generally, we have
\begin{equation}\label{Op1}
\begin{aligned}
&s^{1/2} \int_\Omega e^{2s\varphi(0)} |\partial_t v(0)|^2 dx+s^{1/2} \int_\Omega e^{2s\varphi(0)} |\nabla v(0)|^2 dx+s^{5/2} \int_\Omega e^{2s\varphi(0)} | v(0)|^2 dx\\
&\leq M \int_{0}^T \int_\Omega e^{2s\varphi} |\partial_t^2v-\Delta v+qv|^2 dxdt + M s \int_{0}^T \int_{\Gamma_0} e^{2s\varphi} |\partial_\nu v|^2 d\sigma dt
\end{aligned}
\end{equation}
\end{remark}

\begin{remark}
It is evident that (\ref{Op1}) implies (\ref{Op}), Thus, it suffices to prove Theorem \ref{Car}.
\end{remark}

\begin{proof}

It is sufficient to prove Lemma \ref{Car} in the case where \(q\equiv 0\). Indeed, we assume that we already established
the inequality 
\begin{equation}\label{ine2}
\begin{aligned}
&s^{1/2} \int_\Omega e^{2s\varphi(0)} |\partial_t v(0)|^2 dx+s^{1/2} \int_\Omega e^{2s\varphi(0)} |\nabla v(0)|^2 dx+s^{5/2} \int_\Omega e^{2s\varphi(0)} | v(0)|^2 dx\\
& + s \int_{0}^T \int_\Omega e^{2s\varphi} \left( |\partial_t v|^2 + |\nabla v|^2 \right) dxdt + s^3 \int_{0}^T \int_\Omega e^{2s\varphi} |v|^2 dxdt \\
&\leq M \int_{0}^T \int_\Omega e^{2s\varphi} |\partial_t^2v-\Delta v|^2 dxdt + M s \int_{0}^T \int_{\Gamma_0} e^{2s\varphi} |\partial_\nu v|^2 d\sigma dt\\
&+Ms  \int_\Omega e^{2s\varphi(T)} \left( |\partial_t v(T)|^2 + |\nabla v(T)|^2 \right) dx + Ms^3 \int_\Omega e^{2s\varphi(T)} |v(T)|^2 dx.
\end{aligned}
\end{equation}
Since \(q\in L^{\infty}(\Omega)\), we have
\begin{equation}
|\partial_t^2v-\Delta v|^2\leq |\partial_t^2v-\Delta v+qv-qv|^2\leq 2 |\partial_t^2v-\Delta v+qv|^2+2M^2|v|^2.
\end{equation}
By choosing \(s\) large, we can absorb the term
\[\int_{0}^T \int_\Omega e^{2s\varphi} |v|^2 dxdt\]
into the left-hand side of the Carleman estimate (\ref{ine2}) with \(q=0\).

In order to prove the Carleman estimates, we set \[w=e^{s\varphi}v \quad \mbox{for all } (x,t)\in\Omega\times(0,T).\] 
Then, we introduce the conjugate operator $P$ defined by
\begin{equation}
Pw = e^{s\varphi} (\partial_t^2-\Delta) \left( e^{-s\varphi} w \right).
\end{equation}
Some easy computations give
\begin{equation}
\begin{aligned}
Pw &= \partial_t^2 w - 2s\lambda\varphi\left( \partial_t w \partial_t\psi - \nabla w \cdot \nabla\psi \right) + s^2\lambda^2\varphi^2 w\left( |\partial_t\psi|^2 - |\nabla\psi|^2 \right) - \Delta w \\
&\quad - s\lambda\varphi w\left( \partial_t^2\psi - \Delta\psi \right) - s\lambda^2\varphi w\left( |\partial_t\psi|^2 - |\nabla\psi|^2 \right)-s\partial_t w \\
&= P_1 w + P_2 w + R_1w,
\end{aligned}
\end{equation}
where
\begin{align}
P_1 w=\partial_t^2 w- \Delta w+ s^2\lambda^2\varphi^2 w\left( |\partial_t\psi|^2 - |\nabla\psi|^2 \right),
\end{align}
\begin{align}
P_2 w &= (\alpha - 1)s\lambda\varphi w(\partial_t^2\psi - \Delta\psi) - s\lambda^2\varphi w(|\partial_t\psi|^2 - |\nabla\psi|^2) \notag \\
&\quad - 2s\lambda\varphi(\partial_t w\partial_t\psi - \nabla w\cdot\nabla\psi), 
\end{align}
and
\begin{align}
R_1 w &= -\alpha s\lambda\varphi w(\partial_t^2\psi - \Delta\psi).
\end{align}
Let
\begin{equation}
\alpha \in \left( \frac{2\beta}{\beta + n}, \frac{2}{\beta + n} \right).
\end{equation}
Since we have
\begin{equation}
\int_{0}^T \int_\Omega \left( |P_1 w|^2 + |P_2 w|^2 \right) dxdt + 2 \int_{0}^T \int_\Omega P_1 w P_2 w dxdt = \int_{0}^T \int_\Omega |P w - R_1 w|^2 dxdt, 
\end{equation}
the main part of the proof is then to bound from below the cross-term
\[
\int_{0}^T \int_\Omega P_1 w P_2 w dxdt=\sum_{i,k=1}^3 I_{i,k}.
\]
We calculate the six terms \(I_{i,k}\), \(i,k=1,2,3\) by integrating by parts with respect to  \((x,t)\).

Integrations by part in time give easily
\[
\begin{aligned}
I_{11} &= \int_{0}^T \int_\Omega \partial_t^2 w \left( (\alpha - 1)s\lambda\varphi w (\partial_t^2\psi - \Delta\psi) \right) dxdt \\
&=(\alpha - 1)s\lambda\int_{\Omega}\partial_t w w\varphi (\partial_t^2\psi - \Delta\psi)|_0^Tdx-(\alpha - 1)s\lambda\int_0^T\int_{\Omega} |\partial_tw|^2\varphi (\partial_t^2\psi - \Delta\psi)dxdt\\
&-\frac{(\alpha - 1)s\lambda}{2}\int_{\Omega} |w|^2\partial_t\varphi (\partial_t^2\psi - \Delta\psi)|_0^Tdx+\frac{(\alpha - 1)s\lambda}{2}\int_0^T\int_{\Omega} |w|^2\partial^2_t\varphi (\partial_t^2\psi - \Delta\psi)dxdt\\
&= (\alpha - 1)s\lambda\int_{\Omega}\partial_t w w\varphi (\partial_t^2\psi - \Delta\psi)|_0^Tdx-(\alpha - 1)s\lambda\int_0^T\int_{\Omega} |\partial_tw|^2\varphi (\partial_t^2\psi - \Delta\psi)dxdt\\
&-\frac{(\alpha - 1)s\lambda}{2}\int_{\Omega} |w|^2\partial_t\varphi (\partial_t^2\psi - \Delta\psi)|_0^Tdx+\frac{(\alpha - 1)s\lambda^2}{2}\int_0^T\int_{\Omega} |w|^2\varphi\partial^2_t\psi (\partial_t^2\psi - \Delta\psi)dxdt\\
&+\frac{(\alpha - 1)s\lambda^3}{2}\int_0^T\int_{\Omega} |w|^2|\partial_t\psi|^2 (\partial_t^2\psi - \Delta\psi)dxdt.
\end{aligned}
\]
Similarly, one has
\[
\begin{aligned}
I_{12} &= \int_{0}^T \int_\Omega \partial_t^2 w \left( -s\lambda^2\varphi w \left( |\partial_t\psi|^2 - |\nabla\psi|^2 \right) \right) dxdt \\
&=-s\lambda^2\int_\Omega \partial_t w  w\varphi(|\partial_t\psi|^2 - |\nabla\psi|^2)|_0^Tdx+s\lambda^2\int_0^T\int_\Omega |\partial_tw|^2 \left(\varphi(|\partial_t\psi|^2 - |\nabla\psi|^2)\right)dxdt\\
& +\frac{s\lambda^2}{2}\int_\Omega |w|^2 \partial_t\left(\varphi(|\partial_t\psi|^2 - |\nabla\psi|^2)\right)|_0^Tdx-\frac{s\lambda^2}{2}\int_0^T\int_\Omega |w|^2 \partial^2_t\left(\varphi(|\partial_t\psi|^2 - |\nabla\psi|^2)\right)dxdt\\
&=-s\lambda^2\int_\Omega \partial_t w  w\varphi(|\partial_t\psi|^2 - |\nabla\psi|^2)|_0^Tdx +\frac{s\lambda^2}{2}\int_\Omega |w|^2 \partial_t\left(\varphi(|\partial_t\psi|^2 - |\nabla\psi|^2)\right)|_0^Tdx\\
&+ s\lambda^2 \int_{0}^T \int_\Omega \varphi |\partial_t w|^2 \left( |\partial_t\psi|^2 - |\nabla\psi|^2 \right) dxdt - s\lambda^2 \int_{0}^T \int_\Omega \varphi |w|^2 |\partial_t^2\psi|^2 dxdt \\
&\quad - \left( 2 + \frac{1}{2} \right) s\lambda^3 \int_{0}^T \int_\Omega \varphi |w|^2 |\partial_t\psi|^2 \partial_t^2\psi dxdt + \frac{s\lambda^3}{2} \int_{0}^T \int_\Omega \varphi |w|^2 |\nabla\psi|^2 \partial_t^2\psi dxdt \\
&\quad - \frac{s\lambda^4}{2} \int_{0}^T \int_\Omega \varphi |w|^2 |\partial_t\psi|^2 \left( |\partial_t\psi|^2 - |\nabla\psi|^2 \right) dxdt
\end{aligned}
\]
and
\[
\begin{aligned}
I_{13} &= \int_{0}^T \int_\Omega \partial_t^2 w \left( -2s\lambda\varphi \left( \partial_t w \partial_t\psi - \nabla w \cdot \nabla\psi \right) \right) dxdt \\
&= -s\lambda \int_\Omega |\partial_tw|^2\varphi\partial_t\psi|_0^Tdx+2s\lambda \int_\Omega \partial_tw \varphi\nabla w\cdot\nabla\psi|_0^Tdx
\\
&+ s\lambda \int_{0}^T \int_\Omega \varphi |\partial_t w|^2 \partial_t^2\psi \, dxdt + s\lambda^2 \int_{0}^T \int_\Omega \varphi |\partial_t w|^2 |\partial_t\psi|^2 \, dxdt \\
&\quad + s\lambda \int_{0}^T \int_\Omega \varphi |\partial_t w|^2 \Delta\psi \, dxdt + s\lambda^2 \int_{0}^T \int_\Omega \varphi |\partial_t w|^2 |\nabla\psi|^2 \, dxdt \\
&\quad - 2s\lambda^2 \int_{0}^T \int_\Omega \varphi \partial_t w \partial_t\psi \nabla w \cdot \nabla\psi \, dxdt\\
&-s\lambda\int_0^T\int_{\partial\Omega} \varphi|\partial_t w|^2\partial_\nu \psi d\sigma dt.
\end{aligned}
\]
Furthermore, by Green’s formula and integration by parts, we obtain
\[
\begin{aligned}
I_{21} &= \int_{0}^T \int_\Omega -\Delta w \left( (\alpha - 1)s\lambda\varphi w (\partial_t^2\psi - \Delta\psi) \right) dxdt \\
&= -(1 - \alpha)s\lambda \int_{0}^T \int_\Omega \varphi |\nabla w|^2 (\partial_t^2\psi - \Delta\psi) dxdt \\
&\quad + \frac{(1 - \alpha)}{2} s\lambda^2 \int_{0}^T \int_\Omega \varphi |w|^2 \Delta\psi (\partial_t^2\psi - \Delta\psi) dxdt \\
&\quad + \frac{(1 - \alpha)}{2} s\lambda^3 \int_{0}^T \int_\Omega \varphi |w|^2 |\nabla\psi|^2 (\partial_t^2\psi - \Delta\psi) dxdt\\
&+(1 - \alpha)s\lambda \int_{0}^T \int_{\partial\Omega} \partial_\nu w\varphi w (\partial_t^2\psi - \Delta\psi) d\sigma dt
-(1 - \alpha)s\lambda\int_{0}^T \int_{\partial\Omega}|w|
^2\partial_\nu(\varphi(\partial_t^2\psi - \Delta\psi))d\sigma dt
.
\end{aligned}
\]
On the other hand,
\[
\begin{aligned}
I_{22} &= \int_{0}^T \int_\Omega -\Delta w \left( -s\lambda^2\varphi w \left( |\partial_t\psi|^2 - |\nabla\psi|^2 \right) \right) dxdt \\
&= -s\lambda^2 \int_{0}^T \int_\Omega \varphi |\nabla w|^2 \left( |\partial_t\psi|^2 - |\nabla\psi|^2 \right) dxdt \\
&\quad - \frac{s\lambda^2}{2} \int_{0}^T \int_\Omega \varphi |w|^2 \Delta\left( |\nabla\psi|^2 \right) dxdt \\
&\quad + \frac{s\lambda^3}{2} \int_{0}^T \int_\Omega \varphi |w|^2 \Delta\psi \left( |\partial_t\psi|^2 - |\nabla\psi|^2 \right) dxdt \\
&\quad + \frac{s\lambda^4}{2} \int_{0}^T \int_\Omega \varphi |w|^2 |\nabla\psi|^2 \left( |\partial_t\psi|^2 - |\nabla\psi|^2 \right) dxdt \\
&\quad - s\lambda^3 \int_{0}^T \int_\Omega \varphi |w|^2 \nabla\psi \cdot \nabla\left( |\nabla\psi|^2 \right) dxdt\\
&+s\lambda^2 \int_{0}^T \int_{\partial\Omega} \varphi  w \left( |\partial_t\psi|^2 - |\nabla\psi|^2 \right) d\sigma dt
+\frac{s\lambda^2}{2} \int_{0}^T \int_{\partial\Omega} |\nabla w|^2 \partial_\nu\varphi\left( |\partial_t\psi|^2 - |\nabla\psi|^2 \right) d\sigma dt\\
&-\frac{s\lambda^2}{2}\int_{0}^T \int_{\partial\Omega} \varphi | w|^2 \partial_\nu\left( |\partial_t\psi|^2 - |\nabla\psi|^2 \right) d\sigma dt
.
\end{aligned}
\]
Using the fact that $w|_{\partial\Omega\times(0,T)} = 0$, $\nabla w = (\partial_\nu w)\nu$ and $|\nabla w|^2 = |\partial_\nu w|^2$ on $\partial\Omega\times(0,T)$, we obtain
\[
\begin{aligned}
I_{23} &= \int_{0}^T \int_\Omega -\Delta w \left( -2s\lambda\varphi \left( \partial_t w \partial_t\psi - \nabla w \cdot \nabla\psi \right) \right) dxdt \\
&=-s\lambda \int_\Omega|\nabla w|^2\varphi\partial_t\psi|_0^Tdx\\
&+ s\lambda \int_{0}^T \int_\Omega \varphi |\nabla w|^2 (\partial_t^2\psi - \Delta\psi) dxdt + 2s\lambda^2 \int_{0}^T \int_\Omega \varphi |\nabla\psi \cdot \nabla w|^2 dxdt \\
&\quad - 2s\lambda^2 \int_{0}^T \int_\Omega \varphi \partial_t w \partial_t\psi \nabla w \cdot \nabla\psi dxdt + s\lambda^2 \int_{0}^T \int_\Omega \varphi |\nabla w|^2 \left( |\partial_t\psi|^2 - |\nabla\psi|^2 \right) dxdt \\
&\quad  + 4s\lambda \int_{0}^T \int_\Omega \varphi |\nabla w|^2 dxdt\\
&+2s\lambda \int_{0}^T \int_{\partial\Omega} \partial_\nu w\varphi \partial_tw\partial_t\psi d\sigma dt-
2s\lambda \int_{0}^T \int_{\partial\Omega} \partial_\nu w\varphi \nabla w\cdot\nabla\psi d\sigma dt+
 s\lambda \int_{0}^T \int_{\partial\Omega} \varphi |\partial_\nu w|^2 \nabla\psi \cdot \nu \, d\sigma dt.
\end{aligned}
\]
One easily writes
\[
\begin{aligned}
I_{31} &= \int_{0}^T \int_\Omega s^2\lambda^2\varphi^2 w \left( |\partial_t\psi|^2 - |\nabla\psi|^2 \right) \left( (\alpha - 1)s\lambda\varphi w (\partial_t^2\psi - \Delta\psi) \right) dxdt \\
&= (\alpha - 1)s^3\lambda^3 \int_{0}^T \int_\Omega \varphi^3 |w|^2 (\partial_t^2\psi - \Delta\psi) \left( |\partial_t\psi|^2 - |\nabla\psi|^2 \right) dxdt
\end{aligned}
\]
and
\[
\begin{aligned}
I_{32} &= \int_{0}^T \int_\Omega s^2\lambda^2\varphi^2 w \left( |\partial_t\psi|^2 - |\nabla\psi|^2 \right) \left( -s\lambda^2\varphi w \left( |\partial_t\psi|^2 - |\nabla\psi|^2 \right) \right) dxdt \\
&= -s^3\lambda^4 \int_{0}^T \int_\Omega \varphi^3 |w|^2 \left( |\partial_t\psi|^2 - |\nabla\psi|^2 \right)^2 dxdt.
\end{aligned}
\]
Finally, some integrations by part enable to obtain
\[
\begin{aligned}
I_{33} &= \int_{0}^T \int_\Omega s^2\lambda^2\varphi^2 w \left( |\partial_t\psi|^2 - |\nabla\psi|^2 \right) \left( -2s\lambda\varphi \left( \partial_t w \partial_t\psi - \nabla w \cdot \nabla\psi \right) \right) dxdt \\
&=-s^3\lambda^3\int_\Omega |w|^2\varphi^3 (|\partial_t\psi|^2 - |\nabla\psi|^2)\partial_t\psi|_0^Tdx\\
&+ s^3\lambda^3 \int_{0}^T \int_\Omega \varphi^3 |w|^2 (\partial_t^2\psi - \Delta\psi) \left( |\partial_t\psi|^2 - |\nabla\psi|^2 \right) dxdt \\
&\quad + 2s^3\lambda^3 \int_{0}^T \int_\Omega \varphi^3 |w|^2 \left( \partial_t^2\psi |\partial_t\psi|^2 + 2|\nabla\psi|^2 \right) dxdt \\
&\quad + 3s^3\lambda^4 \int_{0}^T \int_\Omega \varphi^3 |w|^2 \left( |\partial_t\psi|^2 - |\nabla\psi|^2 \right)^2 dxdt\\
&+ s^3\lambda^3\int_0^T\int_{\partial\Omega} \varphi^3 (|\partial_t\psi|^2 - |\nabla\psi|^2)|w|^2\partial_\nu\psi d\sigma dt.
\end{aligned}
\]
Gathering all the terms that have been computed, we get
\begin{equation}\label{P}
\begin{aligned}
\int_{0}^T \int_\Omega P_1 w P_2 w \, dxdt
&= 2s\lambda \int_{0}^T \int_\Omega \varphi |\partial_t w|^2 \partial_t^2\psi \, dxdt - \alpha s\lambda \int_{0}^T \int_\Omega \varphi |\partial_t w|^2 (\partial_t^2\psi - \Delta\psi) dxdt \\
&\quad + 2s\lambda^2 \int_{0}^T \int_\Omega \varphi \left( |\partial_t w|^2 |\partial_t\psi|^2 - 2\partial_t w \partial_t\psi \nabla w \cdot \nabla\psi + |\nabla\psi \cdot \nabla w|^2 \right) dxdt \\
&\quad + 4s\lambda \int_{0}^T \int_\Omega \varphi |\nabla w|^2 dxdt + \alpha s\lambda \int_{-T}^T \int_\Omega \varphi |\nabla w|^2 (\partial_t^2\psi - \Delta\psi) dxdt \\
&\quad + 2s^3\lambda^4 \int_{0}^T \int_\Omega \varphi^3 |w|^2 \left( |\partial_t\psi|^2 - |\nabla\psi|^2 \right)^2 dxdt \\
&\quad + 2s^3\lambda^3 \int_{0}^T \int_\Omega \varphi^3 |w|^2 \left( \partial_t^2\psi |\partial_t\psi|^2 + 2|\nabla\psi|^2 \right) dxdt \\
&\quad + \alpha s^3\lambda^3 \int_{0}^T \int_\Omega \varphi^3 |w|^2 (\partial_t^2\psi - \Delta\psi) \left( |\partial_t\psi|^2 - |\nabla\psi|^2 \right) dxdt + X_1\\
&+(\alpha - 1)s\lambda\int_{\Omega}\partial_t w w\varphi (\partial_t^2\psi - \Delta\psi)|_0^Tdx
-\frac{(\alpha - 1)s\lambda}{2}\int_{\Omega} |w|^2\partial_t\varphi (\partial_t^2\psi - \Delta\psi)|_0^Tdx\\
&-s\lambda^2\int_\Omega \partial_t w  w\varphi(|\partial_t\psi|^2 - |\nabla\psi|^2)|_0^Tdx +\frac{s\lambda^2}{2}\int_\Omega |w|^2 \partial_t\left(\varphi(|\partial_t\psi|^2 - |\nabla\psi|^2)\right)|_0^Tdx\\
&-s\lambda \int_\Omega |\partial_tw|^2\varphi\partial_t\psi|_0^Tdx+2s\lambda \int_\Omega \partial_tw \varphi\nabla w\cdot\nabla\psi|_0^Tdx\\
&-s\lambda \int_\Omega|\nabla w|^2\varphi\partial_t\psi|_0^Tdx\\
&+(1 - \alpha)s\lambda \int_{0}^T \int_{\partial\Omega} \partial_\nu w\varphi w (\partial_t^2\psi - \Delta\psi) d\sigma dt
-(1 - \alpha)s\lambda\int_{0}^T \int_{\partial\Omega}|w|
^2\partial_\nu(\varphi(\partial_t^2\psi - \Delta\psi))d\sigma dt\\
&+s\lambda^2 \int_{0}^T \int_{\partial\Omega} \varphi  w \left( |\partial_t\psi|^2 - |\nabla\psi|^2 \right) d\sigma dt
+\frac{s\lambda^2}{2} \int_{0}^T \int_{\partial\Omega} |\nabla w|^2 \partial_\nu\varphi\left( |\partial_t\psi|^2 - |\nabla\psi|^2 \right) d\sigma dt\\
&-\frac{s\lambda^2}{2}\int_{0}^T \int_{\partial\Omega} \varphi | w|^2 \partial_\nu\left( |\partial_t\psi|^2 - |\nabla\psi|^2 \right) d\sigma dt\\
&
+2s\lambda \int_{0}^T \int_{\partial\Omega} \partial_\nu w\varphi \partial_tw\partial_t\psi d\sigma dt\\
&-
 s\lambda \int_{0}^T \int_{\partial\Omega} \varphi |\partial_\nu w|^2 \nabla\psi \cdot \nu \, d\sigma dt\\
 &
+ s^3\lambda^3\int_0^T\int_{\partial\Omega} \varphi^3 (|\partial_t\psi|^2 - |\nabla\psi|^2)|w|^2\partial_\nu\psi d\sigma dt
\end{aligned}
\end{equation}
where $X_1$ gathers the non-dominating terms and satisfies
\[
|X_1| \leq M s \lambda^4 \int_{0}^{T} \int_{\Omega} \varphi |w|^2 \, dx dt.
\]
When \(t_0>2|x-x_0|\), by Cauchy-Schwarz, we notice that
\begin{equation}
\begin{split}
&s\lambda \int_\Omega (|\partial_tw|^2\varphi\partial_t\psi) (0)dx
+s\lambda \int_\Omega(|\nabla w|^2\varphi\partial_t\psi) (0)dx-2s\lambda \int_\Omega (\partial_tw \varphi\nabla w\cdot\nabla\psi) (0)dx\\
&\geq s\lambda \int_\Omega (|\partial_tw|^2\varphi\big(\partial_t\psi -|\nabla\psi|\big))(0)dx+s\lambda \int_\Omega (|\nabla w|^2\varphi\big(\partial_t\psi -|\nabla\psi|\big))(0)dx\geq 0.
\end{split}
\end{equation}
Since \(\alpha \in \left( \frac{2\beta}{\beta + n}, \frac{2}{\beta + n} \right)\). By explicit computations, we have
\[
2\partial_t^2 \psi - \alpha(\partial_t^2 \psi - \Delta \psi) > 0 \quad \text{and} \quad 4 + \alpha(\partial_t^2 \psi - \Delta \psi) > 0.
\]
As a direct consequence, we can write
\begin{align}
& 2s\lambda \int_{-T}^{T} \int_{\Omega} \varphi |\partial_t w|^2 \partial_t^2 \psi \,dxdt 
- \alpha s\lambda \int_{-T}^{T} \int_{\Omega} \varphi |\partial_t w|^2 (\partial_t^2 \psi - \Delta \psi) \,dxdt \notag \\
& \quad + 4s\lambda \int_{-T}^{T} \int_{\Omega} \varphi |\nabla w|^2 \,dxdt 
+ \alpha s\lambda \int_{-T}^{T} \int_{\Omega} \varphi |\nabla w|^2 (\partial_t^2 \psi - \Delta \psi) \,dxdt \notag \\
& \quad \ge M s\lambda \int_{-T}^{T} \int_{\Omega} \varphi |\partial_t w|^2 \,dxdt 
+ M s\lambda \int_{-T}^{T} \int_{\Omega} \varphi |\nabla w|^2 \,dxdt. 
\end{align}
On the other hand,  we can observe that
\begin{align*}
2s^3\lambda^4 \int_{0}^{T} \int_{\Omega} \varphi^3 |w|^2 \bigl(|\partial_t \psi|^2 - |\nabla \psi|^2\bigr)^2 dxdt
&+ \alpha s^3\lambda^3 \int_{0}^{T} \int_{\Omega} \varphi^3 |w|^2 \bigl(\partial_t^2 \psi - \Delta \psi\bigr)\bigl(|\partial_t \psi|^2 - |\nabla \psi|^2\bigr) dxdt \\
&+ 2s^3\lambda^3 \int_{0}^{T} \int_{\Omega} \varphi^3 |w|^2 \bigl(\partial_t^2 \psi |\partial_t \psi|^2 + 2|\nabla \psi|^2\bigr) dxdt \\
&\geq M s^3\lambda^3 \int_{0}^{T} \int_{\Omega} \varphi^3 |w|^2  dxdt.
\end{align*}
Then combining the above inequality and (\ref{P}), we have
\begin{equation}\label{1p}
\begin{aligned}
&\int_{0}^T \int_\Omega P_1 w P_2 w \, dxdt+2 s\lambda \int_{0}^T \int_{\partial\Omega} \varphi |\partial_\nu w|^2 (x-x_0) \cdot \nu(x) \, d\sigma dt \\
&\geq  M s\lambda \int_{0}^{T} \int_{\Omega} \varphi |\partial_t w|^2 \,dxdt 
+ M s\lambda \int_{0}^{T} \int_{\Omega} \varphi |\nabla w|^2 \,dxdt +M s^3\lambda^3 \int_{0}^{T} \int_{\Omega} \varphi^3 |w|^2  dxdt \\
&+(\alpha - 1)s\lambda\int_{\Omega}\partial_t w w\varphi (\partial_t^2\psi - \Delta\psi)|_0^Tdx
-\frac{(\alpha - 1)s\lambda}{2}\int_{\Omega} |w|^2\partial_t\varphi (\partial_t^2\psi - \Delta\psi)|_0^Tdx\\
&-s\lambda^2\int_\Omega \partial_t w  w\varphi(|\partial_t\psi|^2 - |\nabla\psi|^2)|_0^Tdx +\frac{s\lambda^2}{2}\int_\Omega |w|^2 \partial_t\left(\varphi(|\partial_t\psi|^2 - |\nabla\psi|^2)\right)|_0^Tdx\\
&-s\lambda \int_\Omega |\partial_tw|^2\varphi\partial_t\psi|_0^Tdx+2s\lambda \int_\Omega \partial_tw \varphi\nabla w\cdot\nabla\psi|_0^Tdx\\
&-s\lambda \int_\Omega|\nabla w|^2\varphi\partial_t\psi|_0^Tdx.
\end{aligned}
\end{equation}
Since 
\begin{equation}\label{pr}
\begin{aligned}
\int_{0}^T \int_\Omega |Pw - Rw|^2 dxdt
&\leq 2 \int_{0}^T \int_\Omega |Pw|^2 dxdt + 2 \int_{0}^T \int_\Omega |Rw|^2 dxdt\\
&\leq M \int_{0}^T \int_\Omega |Pw|^2 dxdt + M s^2\lambda^2 \int_{0}^T \int_\Omega \varphi^2 |w|^2 dxdt,
\end{aligned}
\end{equation}
using (\ref{pr}) and (\ref{1p}), we obtain
\begin{equation}\label{Pz}
\begin{aligned}
& M s\lambda \int_{0}^{T} \int_{\Omega} \varphi |\partial_t w|^2 \,dxdt 
+ M s\lambda \int_{0}^{T} \int_{\Omega} \varphi |\nabla w|^2 \,dxdt\\
&\quad +M s^3\lambda^3 \int_{0}^{T} \int_{\Omega} \varphi^3 |w|^2  dxdt + \int_{0}^T \int_\Omega |P_1 w|^2 \, dxdt+\int_{0}^T \int_\Omega |P_2 w|^2 \, dxdt\\
&+(\alpha - 1)s\lambda\int_{\Omega}\partial_t w w\varphi (\partial_t^2\psi - \Delta\psi)|_0^Tdx
-\frac{(\alpha - 1)s\lambda}{2}\int_{\Omega} |w|^2\partial_t\varphi (\partial_t^2\psi - \Delta\psi)|_0^Tdx\\
&-s\lambda^2\int_\Omega \partial_t w  w\varphi(|\partial_t\psi|^2 - |\nabla\psi|^2)|_0^Tdx +\frac{s\lambda^2}{2}\int_\Omega |w|^2 \partial_t\left(\varphi(|\partial_t\psi|^2 - |\nabla\psi|^2)\right)|_0^Tdx\\
&-s\lambda \int_\Omega |\partial_tw|^2\varphi\partial_t\psi|_0^Tdx+2s\lambda \int_\Omega \partial_tw \varphi\nabla w\cdot\nabla\psi|_0^Tdx\\
&-s\lambda \int_\Omega|\nabla w|^2\varphi\partial_t\psi|_0^Tdx\\
&\leq M \int_{0}^T \int_\Omega |P w|^2 dxdt + M s\lambda \int_{-T}^T \int_{\Gamma_0} \varphi |\partial_\nu w|^2 (x - x_0) \cdot \nu(x) \, d\sigma dt.
\end{aligned}
\end{equation}
Next, multiplying \(P_1w\) by \(\partial_tw\) and integrating by parts, we have
\begin{equation}\label{P1}
\begin{split}
&\int_0^T\int_{\Omega}P_1w \partial_twdxdt=\int_0^T\int_{\Omega}\big(\partial_t^2 w- \Delta w+ s^2\lambda^2\varphi^2 w\left( |\partial_t\psi|^2 - |\nabla\psi|^2 \right) \big)\partial_twdxdt\\
&=\frac{1}{2}\int_{\Omega}|\partial_tw|^2(T)dx-\frac{1}{2}\int_{\Omega}|\partial_tw|^2(0)dx+\frac{1}{2}\int_{\Omega}|\nabla w|^2(T)dx-\frac{1}{2}\int_{\Omega}|\nabla w|^2(0)dx \\
&+\frac{1}{2}\int_{\Omega}s^2\lambda^2|w|^2\varphi^2\left( |\partial_t\psi|^2 - |\nabla\psi|^2\right)(T)dx-\frac{1}{2}\int_{\Omega}s^2\lambda^2|w|^2\varphi^2\left( |\partial_t\psi|^2 - |\nabla\psi|^2\right)(0)dx\\
&-\frac{1}{2}\int_0^T\int_{\Omega}s^2\lambda^2|w|^2\partial_t\big(\varphi^2\left( |\partial_t\psi|^2 - |\nabla\psi|^2\right)\big)dxdt.
\end{split}
\end{equation}
Using Cauchy-Schwarz and \(t_0>2|x-x_0|\), we have
\begin{equation}\label{0P1}
\begin{split}
s^{1/2}\int_{\Omega}|\partial_tw|^2(0)dx+s^{1/2}\int_{\Omega}|\nabla w|^2(0)dx+s^{5/2}\int_{\Omega}\lambda^2|w|^2\varphi^2\left( |\partial_t\psi|^2 - |\nabla\psi|^2\right)(0)dx \\
\leq  \int_0^T\int_{\Omega}|P_1w|^2dxdt+s\int_0^T\int_{\Omega}|\partial_tw|^2dxdt+s^{5/2}\int_0^T\int_{\Omega}\lambda^2|w|^2\partial_t\big(\varphi^2\left( |\partial_t\psi|^2 - |\nabla\psi|^2\right)\big)dxdt\\
+2s^{1/2}\int_{\Omega}|\partial_tw|^2(T)dx+2s^{1/2}\int_{\Omega}|\nabla w|^2(T)dx
+2s^{5/2}\int_{\Omega}|w|^2(T)dx.
\end{split}
\end{equation}
Choosing \(s>0\) sufﬁciently large, we can absorb all the terms
\begin{equation}
\begin{aligned}
&(\alpha - 1)s\lambda\int_{\Omega}\partial_t w w\varphi (\partial_t^2\psi - \Delta\psi)(0)dx
-\frac{(\alpha - 1)s\lambda}{2}\int_{\Omega} |w|^2\partial_t\varphi (\partial_t^2\psi - \Delta\psi)(0)dx\\
&-s\lambda^2\int_\Omega \partial_t w  w\varphi(|\partial_t\psi|^2 - |\nabla\psi|^2)(0)dx +\frac{s\lambda^2}{2}\int_\Omega |w|^2 \partial_t\left(\varphi(|\partial_t\psi|^2 - |\nabla\psi|^2)\right)(0)dx\\
&-s\lambda \int_\Omega |\partial_tw|^2\varphi\partial_t\psi|_0^Tdx+2s\lambda \int_\Omega \partial_tw \varphi\nabla w\cdot\nabla\psi (0)dx\\
&-s\lambda \int_\Omega|\nabla w|^2\varphi\partial_t\psi (0)dx.
\end{aligned}
\end{equation}
of (\ref{Pz}) into
\begin{equation}
s^{1/2}\int_{\Omega}|\partial_tw|^2(0)dx+s^{1/2}\int_{\Omega}|\nabla w|^2(0)dx+s^{5/2}\int_{\Omega}\lambda^2|w|^2\varphi^2\left( |\partial_t\psi|^2 - |\nabla\psi|^2\right)(0)dx.
\end{equation}
Then taking \(s\) large, (\ref{Pz}) becomes
\begin{equation}\label{zP}
\begin{aligned}
&s^{1/2}\int_{\Omega}|\partial_tw|^2(0)dx+s^{1/2}\int_{\Omega}|\nabla w|^2(0)dx+s^{5/2}\int_{\Omega}\lambda^2|w|^2\varphi^2\left( |\partial_t\psi|^2 - |\nabla\psi|^2\right)(0)dx\\
&+ M s\lambda \int_{0}^{T} \int_{\Omega} \varphi |\partial_t w|^2 \,dxdt 
+ M s\lambda \int_{0}^{T} \int_{\Omega} \varphi |\nabla w|^2 \,dxdt\\
&\quad +M s^3\lambda^3 \int_{0}^{T} \int_{\Omega} \varphi^3 |w|^2  dxdt + \int_{0}^T \int_\Omega |P_1 w|^2 \, dxdt+\int_{0}^T \int_\Omega |P_2 w|^2 \, dxdt\\
&\leq M \int_{0}^T \int_\Omega |P w|^2 dxdt + M s\lambda \int_{-T}^T \int_{\Gamma_0} \varphi |\partial_\nu w|^2 (x - x_0) \cdot \nu(x) \, d\sigma dt\\
&+Ms\int_{\Omega}|\partial_tw|^2(T)dx+s\int_{\Omega}|\nabla w|^2(T)dx
+s^3\int_{\Omega}|w|^2(T)dx.
\end{aligned}
\end{equation}
Using \(w=e^{s\varphi}v \)
and
\(Pw = e^{s\varphi} (\partial_t^2-\Delta) \left( e^{-s\varphi} w \right)\), we go back to the variable \(v\) in (\ref{zP})
and obtain that there exists some positive constant M such that for all \(s\) and \(\lambda\) large,
\begin{equation}\label{PT}
\begin{aligned}
&s^{1/2} \int_\Omega e^{2s\varphi(0)} |\partial_t v(0)|^2 dx+s^{1/2} \int_\Omega e^{2s\varphi(0)} |\nabla v(0)|^2 dx+s^{5/2} \int_\Omega e^{2s\varphi(0)} | v(0)|^2 dx \\
&+ s \int_{0}^T \int_\Omega e^{2s\varphi} \left( |\partial_t v|^2 + |\nabla v|^2 \right) dxdt + s^3 \int_{0}^T \int_\Omega e^{2s\varphi} |v|^2 dxdt \\
&\leq M \int_{0}^T \int_\Omega e^{2s\varphi} |\partial_t^2v-\Delta v+qv|^2 dxdt + M s \int_{0}^T \int_{\Gamma_0} e^{2s\varphi} |\partial_\nu v|^2 d\sigma dt\\
&+Ms  \int_\Omega e^{2s\varphi(T)} \left( |\partial_t v(T)|^2 + |\nabla v(T)|^2 \right) dx + s^3 \int_\Omega e^{2s\varphi(T)} |v(T)|^2 dx.
\end{aligned}
\end{equation}

When the time \(T\) is large enough in the sense of, we claim that the conditions at times \(T\) can be removed of (\ref{PT}).
More details can be found in \cite{BBE13} and are omitted here. Consequently we have
\begin{equation}\label{PT}
\begin{aligned}
&s^{1/2} \int_\Omega e^{2s\varphi(0)} |\partial_t v(0)|^2 dx+s^{1/2} \int_\Omega e^{2s\varphi(0)} |\nabla v(0)|^2 dx+s^{5/2} \int_\Omega e^{2s\varphi(0)} | v(0)|^2 dx \\
&+ s \int_{0}^T \int_\Omega e^{2s\varphi} \left( |\partial_t v|^2 + |\nabla v|^2 \right) dxdt + s^3 \int_{0}^T \int_\Omega e^{2s\varphi} |v|^2 dxdt \\
&\leq M \int_{0}^T \int_\Omega e^{2s\varphi} |\partial_t^2v-\Delta v+qv|^2 dxdt + M s \int_{0}^T \int_{\Gamma_0} e^{2s\varphi} |\partial_\nu v|^2 d\sigma dt.
\end{aligned}
\end{equation}

Thus the proof of Lemma \ref{Car} is complete.

\end{proof}

\section*{Acknowledgment}
The work of Suliang Si is supported by  the Shandong Provincial Natural Science Foundation (No. ZR2022QA111).


\begin{thebibliography}{100}

\bibitem{BK1981}A. Bukhgeim and M. Klibanov, Global uniqueness of a class of multidimensional inverse
 problems, Sov. Math. Dokl., 24(1981), 244–247.



\bibitem{Beilina2012}
L. Beilina L and M. Klibanov, 
Approximate Global Convergence and Adaptivity for Coefficient Inverse Problems,
Berlin: Springer; 2012.

\bibitem{Bellassoued2017}
M. Bellassoued M and M. Yamamoto M ,
Carleman Estimates and Applications to Inverse Problems for Hyperbolic Systems,
Tokyo: Springer-Japan; 2017.


\bibitem{BBE13} L. Baudouin, M. de Buhan and S. Ervedoza, Global Carleman estimates for waves
and applications, Comm. Partial Diﬀerential Equatrions, 38 (2013), 823–859.


\bibitem{BBE17} L. Baudouin, M. de Buhan and S. Ervedoza, Convergent Algorithm Based on Carleman Estimates for the Recovery of a Potential in the Wave Equation, SIAM J. Number .Anal., 55 (2017), 1578-1613.




\bibitem{Fu2019}
X. Fu, Q. Lü and X. Zhang, Carleman Estimates for Second Order Partial Differential Operators and Applications,
Berlin: Springer; 2019.





\bibitem{HIY} X. Huang,, O. Imanuvilov and M. Yamamoto,  Stability for inverse source problems by Carleman estimates. Inverse Problems, 36(2020), 125006.



\bibitem{Imanuvilov1998}
O. Imanuvilov O and M. Yamamoto, 
Lipschitz stability in inverse parabolic problems by the Carleman estimate,
Inverse Problems, 14(1998), 1229-1245.

\bibitem{Imanuvilov2001}
O. Imanuvilov O and M. Yamamoto,
Global Lipschitz stability in an inverse hyperbolic problem by interior observations, Inverse Problems, 17(2001), 717-728.

\bibitem{Imanuvilov2001a}
O. Imanuvilov O and M. Yamamoto, 
Global uniqueness and stability in determining coefficients of wave equations, Commun. Part. Differ. Equ., 26(2001), 1409-1425.

\bibitem{Imanuvilov2003}
O. Imanuvilov O and M. Yamamoto, 2003
Determination of a coefficient in an acoustic equation with a single measurement, Inverse Problems, 19(2003), 157-171.

\bibitem{Isakov1990}
V. Isakov,
Inverse Source Problems,
 Providence (RI): American Mathematical Society; 1990.

\bibitem{JLY2017} D. Jiang, Y. Liu and M. Yamamoto, Inverse source problem for the hyperbolic equation
with a time-dependent principal part, J. Differ. Equ, 262 (2017), 653–681.

\bibitem{K2002} M. V. Klibanov, Carleman estimates and inverse problems: uniqueness and convexification
of multiextremal objective functions, (2002), 219–252.




\bibitem{Y1995} M. Yamamoto, Stability, reconstruction formula and regularization for an inverse source
hyperbolic problem by a control method, Inverse Probl, 11 (1995), 481–496.

\bibitem{Y1999} M. Yamamoto, Uniqueness and stability in multidimensional hyperbolic inverse problems,
J. Math. Pures Appl., 78 (1999), 65–98.




\bibitem{PY96} J.-P. Puel and M. Yamamoto, On a global estimate in a linear inverse hyperbolic
problem, Inverse Problems, 12 (1996),  995–1002.

\bibitem{PY97} J.-P. Puel and M. Yamamoto, Generic well-posedness in a multidimensional hyperbolic inverse problem, J. Inverse Ill-Posed Probl., 5 (1997),  55–83.


\bibitem{IY01a} O. Y. Imanuvilov and M. Yamamoto, Global Lipschitz stability in an inverse hyper-
bolic problem by interior observations, Inverse Problems, 17 (2001), 717–728.

\bibitem{IY01b} O. Y. Imanuvilov and M. Yamamoto, Global uniqueness and stability in determining
coeﬃcients of wave equations,Comm. Partial Diﬀerential Equations, 26 (2001),
1409–1425.

\bibitem{IY03} O. Y. Imanuvilov and M. Yamamoto, Determination of a coefficient in an acoustic
equation with a single measurement,Inverse Problems, 19 (2003),  157–171.


\bibitem{KY06} M. V. Klibanov and M. Yamamoto, Lipschitz stability of an inverse problem for an
acoustic equation, Appl. Anal., 85 (2006),  515–538.




\bibitem{K60} R. E. Kalman, A new approach to linear filtering and prediction problems, \textit{J. Basic Eng.}, 82 (1960), no. 1, 35–45.

\bibitem{Russell78} D. L. Russell, Controllability and stabilizability theory for linear partial differential equations: recent progress and open questions, \textit{SIAM Rev.}, 20 (1978), no. 4, 639–739.

\bibitem{Ho86} L. F. Ho, Observabilité frontière de l’équation des ondes, \textit{C. R. Acad. Sci. Paris Sér. I Math.}, 302 (1986), no. 10, 443–446.

\bibitem{Lions88} J.-L. Lions, \textit{Contrôlabilité exacte, perturbations et stabilisation des systèmes distribués}, Masson, Paris, 1988.

\bibitem{Komornik94} V. Komornik, \textit{Exact Controllability and Stabilization: The Multiplier Method}, John Wiley \& Sons, Chichester, 1994.

\bibitem{BLR92} C. Bardos, G. Lebeau and J. Rauch, Sharp sufficient conditions for the observation, control, and stabilization of waves from the boundary, \textit{SIAM J. Control Optim.}, 30 (1992), no. 5, 1024–1065.

\bibitem{Yao99} P. F. Yao, On the observability inequalities for wave equations with variable coefficients, \textit{SIAM J. Control Optim.}, 37 (1999), no. 5, 1568–1599.

\bibitem{Carleman39} T. Carleman, Sur un problème d’unicité pour les systèmes d’équations aux dérivées partielles à deux variables indépendantes, \textit{Acta Math.}, 69 (1939), 163–224.

\bibitem{FI94} A. V. Fursikov and O. Y. Imanuvilov, \textit{Controllability of Evolution Equations}, Lecture Notes Ser. 34, Seoul National Univ., Seoul, 1994.

\bibitem{Tataru95} D. Tataru, Boundary observability for conservative PDEs, \textit{Appl. Math. Optim.}, 31 (1995), no. 3, 257–278.

\bibitem{Zuazua93} E. Zuazua, Exact controllability for semilinear wave equations in one space dimension, \textit{Ann. Inst. H. Poincaré Anal. Non Linéaire}, 10 (1993), no. 1, 109–129.

\bibitem{Zhang00a} X. Zhang, Explicit observability inequalities for the wave equation with lower order terms by means of Carleman inequalities, \textit{SIAM J. Control Optim.}, 39 (2000), no. 3, 812–834.

\bibitem{Zhang00b} X. Zhang, Explicit observability estimate for the wave equation with potential and its application, \textit{Proc. R. Soc. Lond. A}, 456 (2000), no. 1997, 1101–1115.

\bibitem{IY01} O. Y. Imanuvilov and M. Yamamoto, Global uniqueness and stability in determining coefficients of wave equations, \textit{Comm. Partial Differ. Equ.}, 26 (2001), no. 7–8, 1409–1425.

\bibitem{Zhang08} X. Zhang, Carleman and observability estimates for stochastic wave equations, \textit{SIAM J. Control Optim.}, 47 (2008), no. 3, 1408–1430.

\bibitem{DZZ08} T. Duyckaerts, X. Zhang and E. Zuazua, On the optimality of observability inequalities for parabolic and hyperbolic systems with potentials, \textit{Ann. Inst. H. Poincaré Anal. Non Linéaire}, 25 (2008), no. 1, 1–41.

\bibitem{Isakov06} V. Isakov, \textit{Inverse Problems for Partial Differential Equations}, 2nd ed., Springer, New York, 2006.

\bibitem{AEWZ14} J. Apraiz, L. Escauriaza, G. Wang and C. Zhang, Observability inequalities and measurable sets, \textit{J. Eur. Math. Soc.}, 16 (2014), no. 11, 2433–2475.




\end{thebibliography}
	\end{document}